\documentclass[11pt]{article}
\usepackage{enumerate}
\usepackage{amsthm,amsmath,amssymb}
\usepackage{graphicx}
\usepackage{lineno}
\usepackage[colorlinks=true,citecolor=black,linkcolor=black,urlcolor=blue]{hyperref}
\usepackage[english]{babel}
\usepackage{amsfonts}
\usepackage{epsfig, subfigure}
\usepackage{amscd,latexsym}
\usepackage{url}
\usepackage{color}
\usepackage{authblk}
\usepackage{subfigure}

\theoremstyle{plain}
\newtheorem{theorem}{Theorem}
\newtheorem{lemma}[theorem]{Lemma}
\newtheorem{cor}[theorem]{Corollary}
\newtheorem{prop}[theorem]{Proposition}

{\itshape}{\rmfamily}

\newtheorem{defi}[theorem]{Definition}

\newcommand{\gp}{\gamma_{\stackrel{}{P}}}

\begin{document}

\title{Zero Forcing sets and Power Dominating sets of cardinality  at most 2}

\author[2]{Najibeh Shahbaznejad\thanks{najibeh.shahbaznejad@uma.ac.ir}}
\author[1]{Ignacio M. Pelayo\thanks{ignacio.m.pelayo@upc.edu}}
\author[2]{Adel P. Kazemi\thanks{adelpkazemi@yahoo.com}}

\affil[1]{Departament de Matem\`atiques, Universitat Polit\`ecnica de Catalunya, Spain}
\affil[2]{Department of Mathematics, University of Mohaghegh Ardabili, Iran}


\maketitle


\begin{abstract}

Let $S$ be a set of vertices  of a graph $G$.
Let $cl(S)$ be the set of vertices built from  $S$,  by iteratively applying  the following propagation  rule:
if a vertex and all but exactly one of its neighbors are in $cl(S)$, then the remaining neighbor is also in $cl(S)$. 
A set  $S$ is called a zero forcing set of $G$ if $cl(S)=V(G)$.
The zero forcing number $Z(G)$ of $G$ is the minimum cardinality of a zero forcing set. 
Let $cl(N[S])$ be the set of vertices built from the closed neighborhood $N[S]$ of  $S$,  by iteratively applying  the previous propagation  rule.
A set  $S$ is called a power dominating set of $G$ if $cl(N[S])=V(G)$.
The power domination number $\gp(G)$ of $G$ is the minimum cardinality of a power dominating set. 
In this paper, we characterize the set of all graphs $G$ for which $Z(G)=2$.
On the other hand, we present a variety of sufficient and/or necessary  conditions for a graph $G$ to satisfy $1 \le \gp(G) \le 2$.

\vspace{+.1cm}\noindent \textbf{Keywords:} zero forcing, domination, power domination, electric power monitoring, maximum nullity.

\vspace{+.1cm}\noindent \textbf{AMS subject classification:} 05C35, 05C69.
\end{abstract}
\vspace{0.5cm}


\section{Introduction}\label{sec1:intro}


This paper is devoted to the study of both the power domination number of connected graphs introduced in \cite{hahehehe02} and the zero forcing number of connected graphs  introduced in \cite{AIM08}. 

The notion of power domination in graphs is a dynamic version of domination where a set of vertices  (power) dominates larger and larger portions of a graph and eventually dominates the whole graph.
The introduction of this parameter  was mainly  inspired by a problem in the electric power system industry \cite{bamiboad93}. 
Electric power networks must be continuously monitored. 
One usual and efficient way of accomplish this monitoring, consist in placing phase measurement units (PMUs), called PMUs,  at selected network locations. 

Due to the high cost of the PMUs, their number must be minimized,  while maintaining the ability to monitor (i.e. to observe)  the entire network. 
The \emph{power domination problem} consists thus  of finding the minimum number of PMUs needed to monitor a given electric power system.
In other words,  a power dominating set of a graph is a set of vertices that observes every vertex in the graph, following the set of rules for power system monitoring described in \cite{hahehehe02}.

Since it was formally introduced  in \cite{hahehehe02}, the power domination number has generated considerable interes; see, for example, \cite{bfffhv18,bfffhvw18,domoklsp08,dovavi16,fehokeyo17,gunira08,koso16,zhkach06}.

The defnition of the power domination number leads naturally to the introduction and  study of the zero forcing number.
As a matter of fact, the zero forcing number of a connected graph $G$ was
introduced in \cite{AIM08} as a tight  upper bound for the maximum nullity of the set of all  real symmetric matrices whose pattern of off-diagonal entries coincides with off-diagonal entries of the adjacency matrix of $G$,
and independently by mathematical physicists studying control of quantum systems \cite{bugi07}.
Since then, this parameter has been extensively investigated; see, for example,  \cite{bfffhvw18,dakast18,erkayi17,geperaso16,gera18,huchye10,kakasu19}.

In this paper, we present a variety of graph families such that all theirs members have either zero fotcing sets or power dominating sets of cardinality at most 2.

\subsection{Basic terminology}

{\small All the graphs considered are undirected, simple, finite and (unless otherwise stated) connected.
Let $v$ be a vertex of a graph $G$.
The \emph{open neighborhood} of $v$ is $\displaystyle N_G(v)=\{w \in V(G) :vw \in E\}$, and the \emph{closed neighborhood} of $v$ is $N_G[v]=N_G(v)\cup \{v\}$ (we will write $N(v)$ and $N[v]$ if the graph $G$ is clear from the context).
The \emph{degree} of $v$ is $\deg(v)=|N(v)|$.
The minimum degree  (resp. maximum degree) of $G$ is $\delta(G)=\min\{\deg(u):u \in V(G)\}$ (resp. $\Delta(G)=\max\{\deg(u):u \in V(G)\}$).
If $\deg(v)=1$, then $v$ is said to be a  \emph{leaf} of $G$.

The distance between vertices $v,w\in V(G)$ is denoted by $d_G(v,w)$, or $d(v,w)$ if the graph $G$ is clear from the context.
The diameter of $G$ is ${\rm diam}(G) = \max\{d(v,w) : v,w \in V(G)\}$.
Let $W\subseteq V(G)$ be a subset of vertices of  $G$.
The  \emph{open neighborhood} of $W$ is $N(W)=\cup_{v\in W} N(v)$ and the  \emph{closed neighborhood} of $W$ is $N[W]=\cup_{v\in W} N[v]$.
Let $u,v \in V(G)$ be  a pair of vertices such that  $d(u,w)=d(v,w)$ for all $w\in V(G)\setminus\{u,v\}$, i.e.,  such that  either $N(u)=N(v)$ or $N[u]=N[v]$. In both cases, $u$ and $v$ are said to be \emph{twins}.

Let $H$ and $G$ be a pair of graphs. 
The graph $H$ is a  \emph{subgraph} of $G$ if it can be obtained from $G$ by removing edges and vertices.
The graph $H$ is an  \emph{induced subgraph} of $G$ if it can be obtained from $G$ by removing vertices.
The subgraph of $G$ induced by a subset of vertices $W$, denoted by $G[W]$, has $W$ as vertex set and $E(G[W]) = \{vw \in E(G) : v \in W,w \in W\}$.
The graph $H$ is a  \emph{minor} of $G$ if it can be obtained from $G$ by removing vertices and by removing and contracting edges.

A set $D$ of vertices of a graph $G$ is a \emph{dominating set} if $N[D]=V(G)$.
The \emph{domination number} $\gamma(G)$ is the minimum cardinality of a dominating set.

Let $K_n$, $K_{r,n-r}$, $S_n\cong K_{1,n-1}$, $P_n$, $W_n$ and $C_n$ denote, respectively, the complete graph, complete bipartite graph, spider, path, wheel and cycle of order $n$.
For undefined terminology and notation,  we refer the reader to \cite{chlezh11}.

The remainder of this paper is organized into two more sections as follows. 
Section 2 is devoted to introducing the zero forcing sets, the zero forzing number $Z(G)$ of a connected graph $G$ and to characterizing the set of all graphs $G$ for which $Z(G)=2$. 
In Section 3, which is subdivided into  three subsections, power dominating sets and the power domination number $\gp(G)$ of a connected graph $G$ are first introduced and then, in the remaining subsections the problem $1 \le \gp(G) \le 2$ is approached from several perspectives.
In Subsection 3.1, a brief list of basic know and new results are shown. 
Next, in Subsection 3.2, some contributions involving  graphs with high maximum degree are presented,
Finally, the mentioned problem $1 \le \gp(G) \le 2$ is investigated in Subsection 3.3 for two binary operations: the lexicographic product and the Cartesian product.

\newpage
\section{Zero forcing number}

The concept of \emph{zero forcing} can be described via the following coloring game on the vertices of a given graph $G=(V,E)$. 
Let $U$ be a proper subset of $V$.
The elements of $U$ are colored black, meanwhile the vertices of $W=V\setminus U$ are colored white.
The color change rule is: 

\begin{center}
{\bf \small If $u \in U$ and exactly one neighbor $w$ of $u$ is white, then change the color of $w$ to black.}
\end{center}

In such a case, we  denote this by $u \rightarrow w$, and we say, equivalentely, that   $u$ forces $w$, that $u$ is a forcing vertex of $w$ and also that $u \rightarrow w$ is a force. 
The \emph{closure} of $U$, denoted $cl(U)$, is the set of black vertices obtained after the color change rule is applied until no new
vertex can be forced; it can be shown that $cl(U)$ is uniquely determined by $U$ (see \cite{AIM08}).

\begin{defi}[\cite{AIM08}]
A subset of vertices $U$ of a graph $G$ is called a \emph{zero forcing set} of $G$ if $cl(U)=V(G)$. 
\end{defi}

A \emph{minimum zero forcing set}, a \emph{ZF-set} for short,  is a zero forcing set of minimum cardinality. 
The \emph{zero forcing number} of $G$ , denoted by $Z(G)$, is the cardinality of a ZF-set.

A \emph{chronological list of forces} ${\cal F}_U$ associated with a  set $U$ is a sequence
of forces applied to obtain $cl(B)$ in the order they are applied. 
A \emph{forcing chain} for the chronological list of forces ${\cal F}_U$ is a maximal sequence of vertices $(v_1, . . . , v_k )$ such that
the force $v_i \rightarrow v_{i+1}$   is in ${\cal F}_U$ for $1 \le i  \le k-1$.
Each forcing chain induces a distinct path in $G$, one of whose endpoints is
in $U$; the other is called a terminal.
Notice that a zero forcing chain can consist of a single vertex $(v_1)$, and this happens   if $v_1 \in U$ and $v_1$ does not perform a force.
Observe also that any two forcing chains are disjoint.

For example, if we consider the graph $G$ shown  in  Figure \ref{fig2}, and take the set $U=\{u_1,u_2,u_3\}$, 
then $cl(U)=\{u_1,u_2,u_3,w_2,w_1,w_5\}$,  ${\cal F}_U=\{u_2 \rightarrow w_2,u_1 \rightarrow w_1,w_1 \rightarrow w_5\}$ and thus the list of forcing chains is: $\{(u_1,w_1,w_5),(u_2,w_2),(u_3)\}$.

\begin{figure}[!h]
	\centerline{\includegraphics[height=5cm]{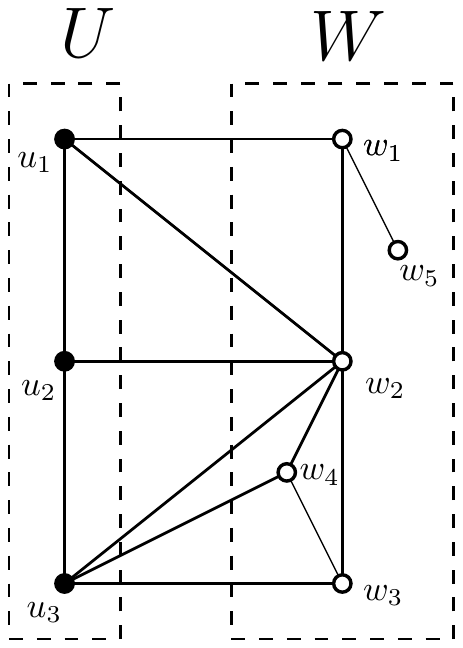}}
	\caption{$V(G)=U \cup W = \{u_1,u_2,u_3\}\cup\{w_1,w_2,w_3,w_4\}$.}
	\label{fig2}
\end{figure}

\begin{prop}[\cite{erkayi17}] 
Let $G$ be a graph of order $n$.
Then,  $Z(G)=1$ if and only if $G$ is the  path $P_n$.
\end{prop}

A graph is outerplanar if it has a crossing-free embedding in
the plane such that all vertices are on the same face.
The \emph{path cover number} $P(G)$ of a graph $G$ is the smallest positive integer $k$ such
that there are $k$ vertex-disjoint induced paths $P_1, \ldots, P_k$ in $G$  that cover all
the vertices of $G$, i.e., $\displaystyle V(G) = \bigcup _{i=1}^k V(P_i)$.

\begin{prop}[\cite{bbfhhsh10}] 
For any graph  $G$, $P(G) \le Z(G)$.
\end{prop}

\newpage
\begin{theorem}\label{mainzf}
Let $G$ be a graph of order $n\ge5$.
Then,  $Z(G)=2$ if and only if $G$ is an outerplanar graph with $P(G)=2$.
\end{theorem}
\begin{proof}
($\Rightarrow$):
Let $S_0=\{u_0,v_0\}$ be a ZF-set of $G$. 
Let ${ \cal F}_{S_0}$ a chronological list of forces associated with $S_0$.
Let $(u_0,u_{i_1},\ldots,u_{1_r})$ and $(v_0,v_{j_1},\ldots,v_{j_s})$ be the pair of forcing chains for ${ \cal F}_{S_0}$, where the indices have been assigned having into account the order of the forces of ${ \cal F}_{S_0}$ (see Figure \ref{fch}, for some examples). 
Notice that $P(G)=2$, since both $G[\{u_0,u_{i_1},\ldots,u_{1_r}\}]$ and $G[\{v_0,v_{j_1},\ldots,v_{j_s}\}]$ are paths.

\begin{figure}[!h]
	\centerline{\includegraphics[height=5cm]{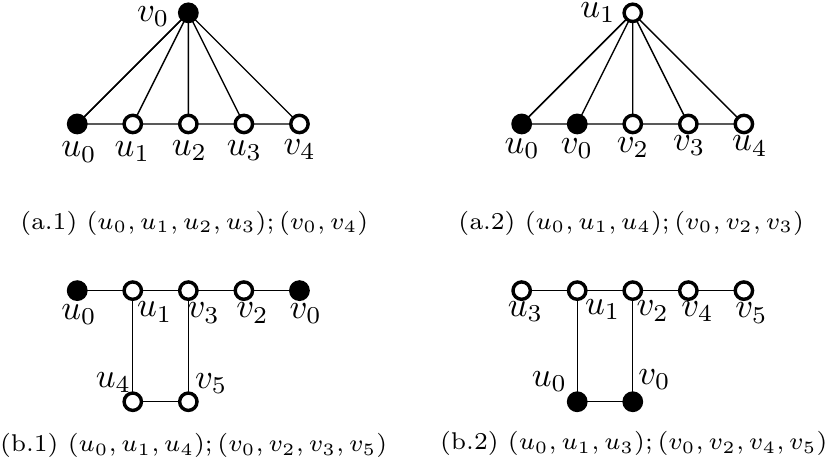}}
	\caption{In all cases, $S_0  = \{u_0,v_0\}$ is a ZF-set.}
	\label{fch}
\end{figure}

Next, we embed this graph in the plane in such a way that 
for every $h\in\{0,1,\ldots,r\}$ and $k\in\{0,1,\ldots,s\}$,   $u_{i_h}=(i_h,0)$ and $v_{j_k}=(j_k,1)$, where $i_0=j_0=0$ (see some examples in Figure \ref{fch3}).

\begin{figure}[!h]
	\centerline{\includegraphics[height=4.5cm]{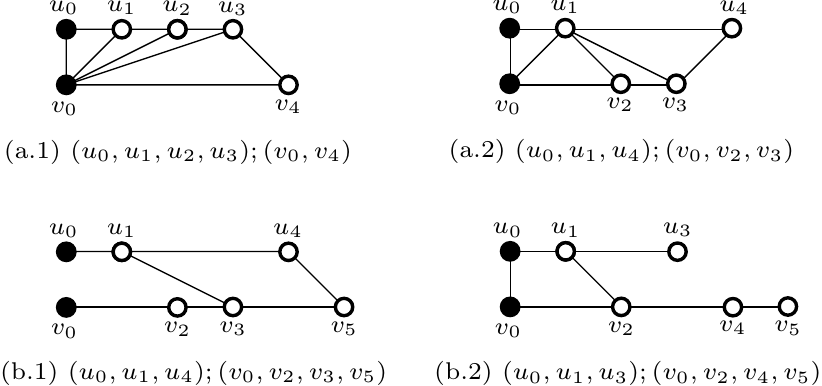}}
	\caption{Some embeddings in the plane.}
	\label{fch3}
\end{figure}

Finally, we prove that $G$ is an outerplanar graph, by showing that if we draw all the edges of $G$, then no two of them intersect.
Take $a,b \in \{0,1,\ldots,r\}$ and $c,d \in \{0,1,\ldots,s\}$ such that $a<b$ and $c<d$ 
and consider the vertices $u_{i_a}$, $u_{i_b}$, $v_{j_c}$ and $v_{j_d}$.
We distinguish six cases (see Figure \ref{fch1}):

\vspace{.2cm}\noindent{\bf Case 1:} If $i_a \le i_b<j_c<j_d$, then $u_{i_a}v_{j_d}\not\in E(G)$.

\vspace{.2cm}\noindent{\bf Case 2:} If $i_a \le j_c<i_b<j_d$, then $u_{i_a}v_{j_d}\not\in E(G)$.

\vspace{.2cm}\noindent{\bf Case 3:} If $i_a \le j_c<j_d<i_b$, then $v_{j_c}u_{i_b}\not\in E(G)$.

\vspace{.2cm}\noindent{\bf Case 4:} If $j_c \le j_d < i_a <i_b $, then $v_{j_c}u_{i_b}\not\in E(G)$.

\vspace{.2cm}\noindent{\bf Case 5:} If $j_c \le i_a < j_d <i_b $, then $v_{j_c}u_{i_b}\not\in E(G)$.

\vspace{.2cm}\noindent{\bf Case 6:} If $j_c \le i_a <i_b < j_d$, then $u_{i_a}v_{j_d}\not\in E(G)$.

\begin{figure}[!h]
	\centerline{\includegraphics[height=4.8cm]{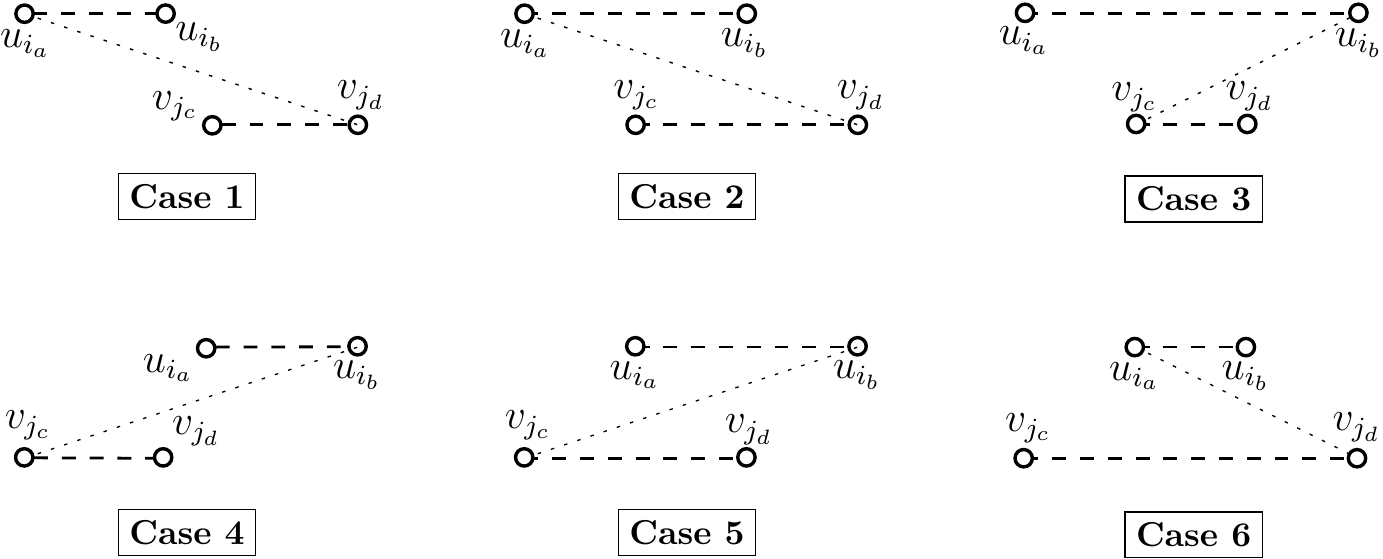}}
	\caption{In all cases, the dotted edge is not possible, as $G$ is outerplanar.}
	\label{fch1}
\end{figure}

\vspace{.3cm}\noindent ($\Leftarrow$):
Recall that a graph of order at least 5 is outerplanar if and only if  it contains neither $K_4$ nor $K_{2,3}$ as a minor.
Let $P_1$, $P_2$  two vertex-disjoint induced paths of $G$ such that $V(G=V(P_1)\cup V(P_2)=\{x_0,\ldots,x_r\}\cup\{y_0,\ldots,y_s\}$ and $E(P_1)\cup E(P_2)=\{x_0x_1,\ldots,x_{r-1}x_r,y_0y_1,\ldots y_{s-1}y_s\} \subsetneq E(G)$.

Next, we embed this graph in the plane as follows.
The  path $P_1$ is an horizontal segment being the left endpoint vertex $x_0$, and the  path $P_2$ is another  horizontal segment parallel to the first one whose left endpoint is vertex $y_0$.
Now, we draw all the edges joining vertices from both paths.
We call this drawing ${\cal D}_1$.
Assumme that no two edges cross in ${\cal D}_1$.
Then, it is a routine exercise to prove that the set $\{x_0,y_0\}$ is zero forcing set.

Suppose, on the contrary,  that there are four integers $i,j,h,k$ such that 
$0 \le  i < j \le  r$, $0  \le h < k \le s$, $x_iy_k, x_jy_h \in E(G)$. 
Then, we embed this graph as follows.
The path $P_1$ is an horizontal segment being the left endpoint vertex $x_0$, meanwhile path $P_2$ is another  horizontal segment paralel to the first one whose  left endpoint is vertex $y_s$.
Now, we draw all the edges joining vertices from both paths.
We call this drawing ${\cal D}_2$ (see Figure \ref{fig5}(a)).
Notice that in this second drawing,  edges $x_iy_k$, $x_jy_h$ do not cross.
We claim that in  ${\cal D}_2$ no two edges cross.
To prove this, we distinguish cases.

\begin{figure}[!h]
	\centerline{\includegraphics[height=9cm]{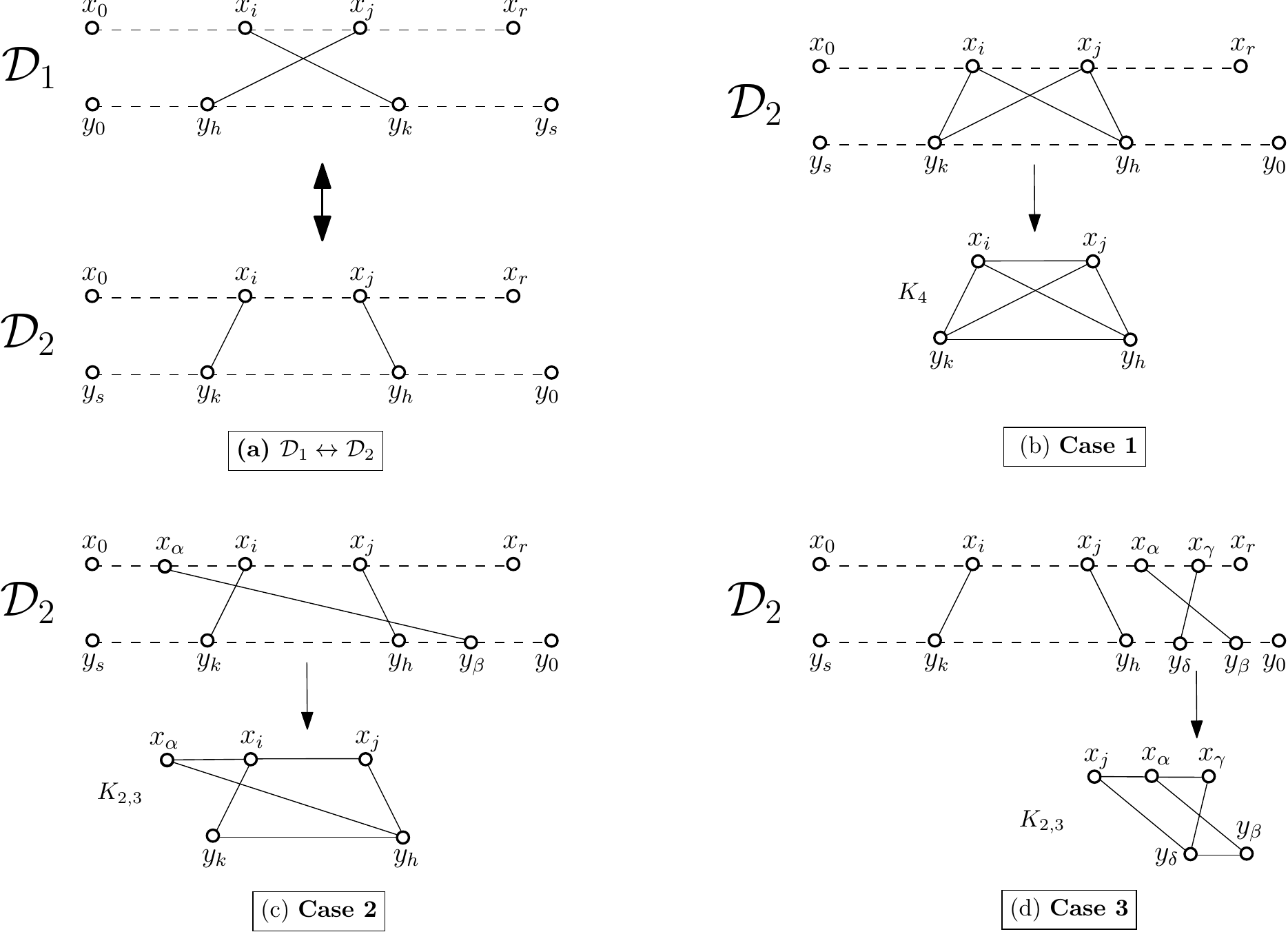}}
	\caption{${\cal D}_1$ and ${\cal D}_2$ are two different embeddings of $G$.}
	\label{fig5}
\end{figure}

\vspace{.2cm}\noindent{\bf Case 1:} 
$x_iy_h, x_jy_k \in E(G)$ (see Figure \ref{fig5}(b)). 
Then, $K_{4}$ is a minor of $G$, a contradiction.

\vspace{.2cm}\noindent{\bf Case 2:} There is an edge $e$ crossing either edge $x_iy_k$ or edge $x_jy_h$ (see Figure \ref{fig5}(c)).
We can suppose w.l.o.g. that $e=x_{\alpha}y_{\beta}$, with $\alpha<i$ and $\beta < h$.
In this case,  $K_{2,3}$ is a minor of $G$, a contradiction.

\vspace{.2cm}\noindent{\bf Case 3:} 
There are two edges $e$ and $e'$, other than $x_iy_k$ and $x_jy_h$, crossing each other (see Figure \ref{fig5}(d)).
In this case,  $K_{2,3}$ is a minor of $G$, a contradiction.
%
\end{proof}

\section{Power domination number}

Zero forcing is closely related to power domination, because power
domination can be described as a domination step followed by the zero forcing process or, equivalentely,  zero forcing can be described as power domination without the domination step.
In other words, the power domination process on a graph $G$ can be described
as choosing a set $S \subset V (G)$ and applying the zero forcing process to the closed neighbourhood $N[S]$ of $S$. 
The set $S$ is thus a power dominating set of $G$ if and only if $N[S]$ is a zero forcing set for $G$

\begin{defi}[\cite{hahehehe02}]{\rm
A subset of vertices $S$ of a graph $G$ is called a \emph{power dominating set} of $G$ if $cl(N[S])=V(G)$. }
\end{defi}

\noindent A \emph{minimum power dominating set}, a \emph{PD-set} for short,  is a power dominating set of minimum cardinality. 
The \emph{power dominating number} of $G$ , denoted by $\gp(G)$, is the cardinality of a PD-set.

\subsection{Basic Results}

\noindent As a straight consequence of these definitions,  it is derived both  that $\gp(G) \le Z(G)$ and $\gp(G) \le \gamma(G)$.
Moreover, this pair of inequalities along with Theorem \ref{mainzf}, allow us to derive the following results.

\begin{cor}
Let $G$ be a graph of order $n$.
\begin{itemize}
\item If $G$ is outerplanar and $P(G)=2$, then $\gp(G) \le 2$.
\item   $\Delta(G)=n-1$ if and only if $\gp(G) = \gamma(G)=1$.
\end{itemize}
\end{cor}

We end this section by presenting a first list of new and know results involving this parameter along with a Table containing  information of some basic graph families.

\begin{prop}
If $G$ is a connected graph of order al most 5, then  $\gp(G)=1$.
Moreover,
\begin{itemize}
	\item The smallest connected graph $G$  such that $\gp(G)=2$ is the H-graph (see Figure \ref{HW8} (a)).
	\item The smallest connected graph $G$  with no twin vertices such that $\gp(G)=2$ is the Wagner graph (see Figure \ref{HW8} (b)).
	\end{itemize}
\end{prop}

\begin{figure}[!h]
	\centerline{\includegraphics[height=4cm]{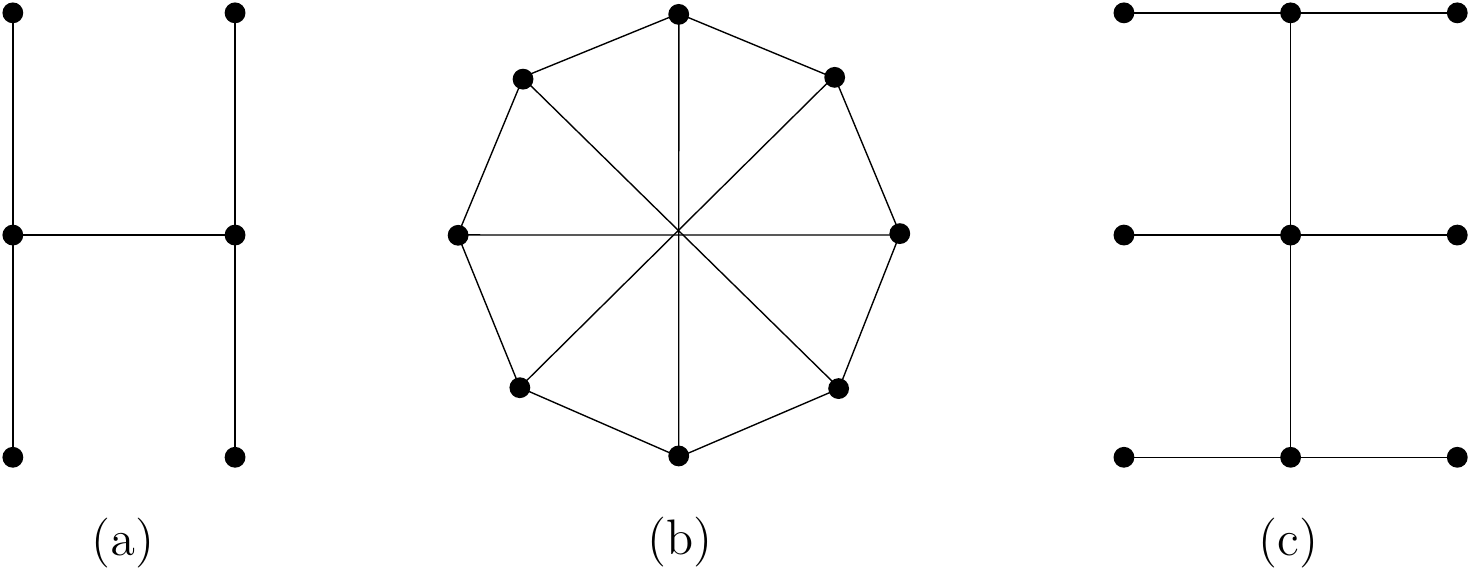}}
	\caption{Some small graphs}
	\label{HW8}
\end{figure}

\begin{table}[ht]
\begin{center}

\begin{tabular}{|c|cccccccc|} \hline
$G$   &   $P_n$     &  $C_n$   & $K_n$ &   $K_{1,n}$ &   $K_{2,n}$ &   $K_{h,n-h}$ &  $W_n$  & \\ \hline \hline
$\gp(G)$ & 1   &    1   &   1  &   1  &   1   &  2   & 1  &       \\ \hline
$\gamma(G)$ & {\scriptsize$\lfloor\frac{n+2}{3}\rfloor$} & {\scriptsize$\lfloor\frac{n+2}{3}\rfloor$} &    1   & $1$   & 2 & 2 & 1 &    \\  \hline
$Z(G)$      &  1 & 2 &    $n-1$   & $n-2$   & $n-2$ & $n-2$ & 3 &    \\  \hline

\end{tabular}
 \caption{Power domination, domination  and zero forcing numbers of some basic graph families.}
 \label{lio}
 \end{center}
\end{table}

\newpage

\begin{prop}
	Let $G = K_{r_{1},\cdots,r_k}$ be the complete $k$-partite graph  with  $2\le k$ and  $1 \le r_{1} \leq r_{2} \leq \cdots \leq r_{k} $ and $V(G)=\cup_{i=1}^kV_i$.
	Let $G_e$ the graph obtained from $G$ by deleting an edge $ e=vw \in E(G)$.
	Then
	\begin{enumerate}[{\rm (1)}] 
	
	\item If $ r_{1} \leq 2 $, then $ \gp(G)=1 $. 
	
    \item If $ r_{1} \geq 3 $, then $ \gp(G)=2 $.
    
	\item If $r_1\le 2$, then $ \gp(G_e)=1$.
	
	\item If $r_1=3$, then 
	$\gp(G_e)=\left\{\begin{array}{lll}
    1, & {\rm if} & \{v,w\}\cap V_1\neq \emptyset \\
      
    2, &  & otherwise. \\
\end{array}\right.$

	\item If $4\le r_1$, then $ \gp(G_e)=2$.
	
	\end{enumerate}
\end{prop}
\begin{proof}
\begin{enumerate}[{\rm (1)}] 
	
	\item Take $v_1\in V_1$. 
	Notice that $N[v_1]=V(G)\setminus [V_1-v_1]$.
	If $r_1=1$, then $\{v_1\}$ is a dominating set of $G$, i.e., $\gp(G)=1$.
	Suppose that $r_1=2$ and $V_1=\{v_1,v_1'\}$.
	Then, for any vertex $u\not\in V_1$, $u \rightarrow v_1'$, which means that $\gp(G)=1$, as $N[v_1]=V(G)\setminus \{v_1'\}$.
	
	\item For every $u\in V_i$, $cl(\{w\})=V(G)\setminus [V_i-u]$.
	Thus, $\gp(G)\ge2$.
	Take $S=\{v_1,v_2\}$, where $v_1\in V_1$ and $v_2\in V_2$.
	Notice that $N[S]=V(G)$.
	Hence, $\gp(G)=\gamma(G)=2$.

	\item If $\{v,w\}\cap V_1=\emptyset$, then proceed as in item (1).
	Suppose w.l.o.g. that $v\in V_1$.
	Notice that $N[v]=V(G)\setminus [(V_1-v)\cup \{w\}]$.
	If $r_1=1$, then for every $u\not\in\{v,w\}$, $u\rightarrow w$.
	Thus, $\gp(G_e)=1$.
	Otherwise, suppose that $r_1=2$ and $V_1=\{v,v'\}$.
	Then, $N[v']=V(G)-v$ and for any vertex $u\not\in \{v,v',w\}$, $u \rightarrow v$.
	Hence, $\gp(G_e)=1$.
	
	\item If $\{v,w\}\cap V_1=\emptyset$, then proceed as in item (2).
	Otherwise, suppose w.l.o.g. that $v\in V_1$ and $V_1=\{v,v',v''\}$.
	Notice that $N[v']=V(G)\setminus \{v,v''\}$.
	Next, observe that $w \rightarrow v''$ and for any vertex  $u\not\in \{v,v',v'',w\}$,   $u  \rightarrow v$.
	Hence, $\gp(G_e)=1$.

	\item Notice that, for every $u\in V(G)$, $cl(u)=N[u]$ and $|N[u]|\le n-3$.
	Thus, $\gp(G)\ge2$.
	Moreover, for every pair of vertices $\{u_1,u_2\}$ such that $\{u_1,u_2\}\cap\{v,w\}=\emptyset$, $N[\{u_1,u_2\}]=V(G)$.
	Hence, $\gp(G_e)=\gamma(G_e)=2$.
	
\end{enumerate}
\vspace{-.4cm}\end{proof}

\vspace{.4cm}
A tree is called a \emph{spider} if it has a unique vertex of degree greater than 2.  
We define the \emph{spider number}  of a tree $T$,  denoted  by $sp(T )$, to be the minimum number of subsets into which $V (T )$ can be partitioned so that each subset induces a spider. 

\vspace{.2cm}
\begin{theorem} [\cite{hahehehe02}]\label{sp}
	For any tree $T$,  $\gp(T) = sp(T)$.
\end{theorem}

\begin{cor} 
	For any tree $T$, $\gp(T) = 1$ if and only if $T$ is a spider.
\end{cor}

\begin{theorem}[\cite{zhka07}]
If $G$ is a planar (resp. outerplanar) graph of diameter at most 2 (resp. at most 3), then $\gp(G)\le 2$ (resp. $\gp(G)=1$).
\end{theorem}

\newpage
\subsection{Graphs with large maximum degree}\label{md}

\begin{prop}\label{gpn-3}
Let $G$ a graph of order $n$ and maximum degree $\Delta$.
\begin{enumerate}[{\rm (1)}]

\item If $n-2\le \Delta \le n-1$, then  $\gp(G)=1$.

\item If $n-4\le \Delta \le n-3$, then  $1 \le \gp(G) \le 2$.

\end{enumerate}
 
\end{prop}
\begin{proof} Let $u$ a vertex such that $deg(u)=\Delta$, that is, such that $|N[u]|=\Delta+1$.

\vspace{.2cm}{\rm (1)} If $\Delta = n-1$, then $1 \le  \gp(G) \le \gamma(G)=1$, which means that $\gp(G)=1$.
Let $u$ a vertex such that $deg(u)=\Delta$, that is, such that $|N[u]|=\Delta+1$.
If $\Delta = n-2$, then $|N[u]|=n-1$, i.e., there exists  a vertex $w$ such that $V(G)\setminus N[u]=\{w\}$.
Thus, for some vertex $v \in N(u)$, $v \rightarrow w$, which means that $\{u\}$ is a PD-set.

\vspace{.2cm}{\rm (2)} 
Suppose that $\Delta = n-3$.
Let $w_1,w_2\in V(G)$ such that $V(G)\setminus N[u]=\{w_1,w_2\}$.
Take the set $S=\{u,w_1\}$.
If $w_1w_2\in E(G)$, then  $S$ is a dominating set of $G$, and thus it is a power dominating set.
If $w_1w_2\not\in E(G)$, then $N[S]=V(G)\setminus \{w_2\}$.
Hence, $S$ is a power dominating set since for some vertex $v \in N(u)$, $v \rightarrow w_2$.

Finally, assume that $\Delta = n-4$.
Let $w_1,w_2,w_3\in V(G)$ such that $V(G)\setminus N[u]=\{w_1,w_2,w_3\}$.
We distinguish cases.

\vspace{.2cm}\noindent{\bf Case 1:} 
$G[S]$ is not the empty graph $\overline{K}_3$.
Suppose w.l.o.g. that $w_1w_2\in E(G)$.
Take the set $S=\{u,w_1\}$.
If $w_1w_3\in E(G)$, then  $S$ is a dominating set of $G$, and thus it is a power dominating set.
If $w_1w_3\not\in E(G)$, then $N[S]=V(G)\setminus \{w_3\}$.
Hence, $S$ is a power dominating set since either $w_2 \rightarrow w_3$ or, for some vertex $v \in N(u)$, $v \rightarrow w_3$.

\vspace{.2cm}\noindent{\bf Case 2:} 
$G[S]$ is  the empty graph $\overline{K}_3$.
For $i\in \{1,2,3\}$, let $v_i\in N(u)$ be such that $v_iw_i\in E(G)$.
If for every $i\in\{1,2,3\}$, $N(v_i)\cap \{w_1,w_2,w_3\}=\{w_i\}$,  then  $\{u\}$ is a dominating set of $G$, and thus it is a power dominating set.
If for some  $i\in\{1,2,3\}$, $|N(v_i)\cap \{w_1,w_2,w_3\}|\ge 2$, assume w.l.o.g. that $i=1$.
In this case, $S=\{u,v_1\}$ is a power dominating set since 
$V(G)\setminus \{w_3\} \subseteq N[S]$ and either $v_1 \rightarrow w_3$ or  $v_3 \rightarrow w_3$.
\end{proof}

There are graphs with maximum degree $\Delta=n-5$ such that $\gp(G)\ge3$. 
The simplest example is shown in Figure \ref{HW8} (c).

\begin{lemma}
	Let $G$ be a  graph of order $n\ge4$.
	Let $u,w_1,w_2 \in V(G)$ such that $deg(u)=n-3$ and $V(G)= N[u] \cup \{w_1,w_2\}$
	Then, $\{u\}$ is a PD-set if and only if  $w_1$ and $w_2$ are not twins.
\end{lemma}
\begin{proof}
Suppose first that $w_1$ and $w_2$ are  twins.
In this case, every power dominating set must contain either $w_1$ or $w_2$.
Conversely, assume that $w_1$ and $w_2$ are not twins.
If $N(w_1)=\{w_2\}$, then for some vertex $v\in N(u)$, $v \rightarrow w_1$ and $w_1 \rightarrow w_2$, which means that $\{u\}$ is a PD-set.
If $deg(w_1)\ge 2$, then take  a vertex $v_1\in N(u)$ such that $w_1\in N(v_1)$ and $w_2 \not\in N(v_1)$.
Thus, $v_1 \rightarrow w_1$ and $v_2 \rightarrow w_2$, for any vertex $v_2$ such that $w_2\in N(v_2)$.
\end{proof}

\begin{cor}
Let $G$ be a  graph of order $n\ge4$.
If there exists a vertex $u\in V(G)$ such that $deg(u)=n-3$ and the pair of vertices of $V(G)\setminus N[u]$ are not twins, then $\gp(G)=1$.
\end{cor}

The converse of this statement is not true.
For example, if we consider the graph $G$ displayed in Figure \ref{G9}, then it is easy to check that $\{w_1\}$ is a PD-set of $G$.

\begin{figure}[!h]
	\centerline{\includegraphics[height=5cm]{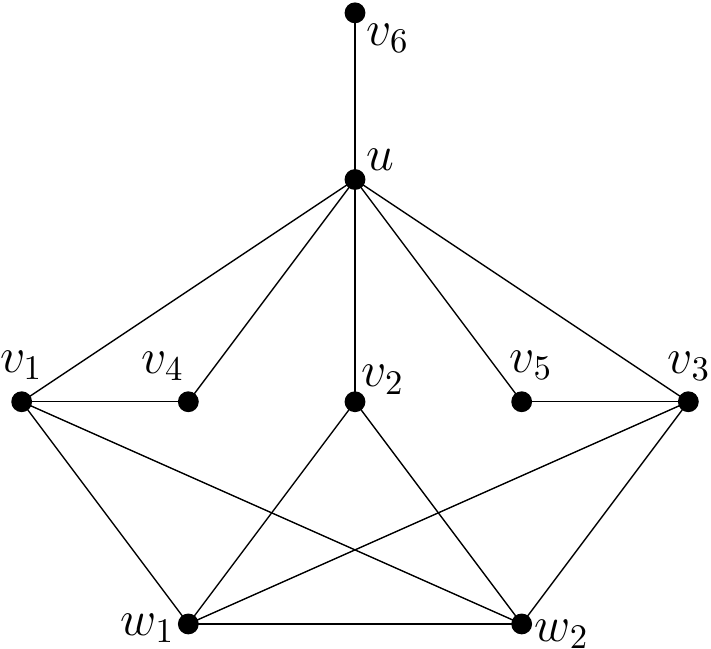}}
	\caption{The list a forcing chains for ${\cal F}_{N[w_1]}$ is: $\{(w_1),(w_2),(v_2,u,v_6),(v_1,v_4),(v_3,v_5)\}$.}
	\label{G9}
\end{figure}

\begin{theorem}
Let $ G $ be a  $ (n-3) $-regular graph of order $n \geq 5$. 
Then, $ \gamma_{p}(G) = 1 $ if and only if there exist an edge $e=uv\in E(G) $ such that $ \mid N[v] \setminus N[u] \mid=1$.
\end{theorem}
\begin{proof}
If $n=5$, then $G\cong C_5$, and the equivalence  is obvious.
Suppose thus that $n\ge6$.
\newline ($\Rightarrow$):
Let $S=\{u\}$ be a $\gp$-set of $G$.
Let $W=V(G)\setminus N[u]=\{x,y\}$.
As $S$ is a $\gp$-set, there must exist a vertex $v\in N(u)$ such that $|N(v)\cap W|=1$.
Hence,   there exist a unique vertex $w\in N(u)\setminus \{v\}$ such that $w\not\in N(v)$, as $\deg(v)=n-3$.
\newline ($\Leftarrow$):
Take the sets $S=\{u\}$ and $W=V(G)\setminus N[u]=\{x,y\}$.
As $ \mid N[v] \setminus N[u] \mid=1$ and $deg(u)=deg(v)=n-3$, $|N(v)\cap W|=1$.
Hence, if for example $N(v)\cap W=\{x\}$, then $v \rightarrow x$, which means that $S$ is a $\gp$-set of $G$.
\end{proof}

\newpage
\subsection{Graph operations}\label{graphoperations}



The vertex set of the \emph{lexicographic product} $G \circ H$ of graphs $G$ and $H$ is $V(G) \times V (H)$. 
Let  $u = (g, h)$ and $v = (g', h')$ be a pair of vertices of $V (G) \times V (H)$.
Vertices $u$ and $v$ are adjacent in the \emph{lexicographic product} $G \circ H$ if either $gg' \in E(G)$, or $g = g'$ and $hh' \in E(H)$.

A set $D$ of vertices of a graph $G$ is a \emph{total dominating set} if $N(D)=V(G)$.
The \emph{domination number} $\gamma_t(G)$ is the minimum cardinality of a total dominating set.

\begin{theorem} [\cite{domoklsp08}]\label{lp_td}
	For any pair of nontrivial connected graphs $G$ and $H$, 
	$$\displaystyle \gamma_{\stackrel{}{P}}(G \circ H)=
\left\{
\begin{array}{lll}
      \gamma(G), & {\rm if} & {\displaystyle \gamma_{\stackrel{}{P}}(H)=1} \\
      
      \gamma_{\stackrel{}{t}}(G), & {\rm if} & {\displaystyle \gamma_{\stackrel{}{P}}(H)>1} \\
\end{array} 
\right. 
$$
\end{theorem}

\begin{theorem} [\cite{hawa03}]\label{td2}
	For any pair  connected graph $G$, 
	$\displaystyle \gamma_{\stackrel{}{t}}(G)=2$ if and only if $diam(\overline{G})>2$
\end{theorem}

\begin{cor} 
	For any pair of nontrivial connected graphs $G$ and $H$, 
	
	\begin{itemize}
	
	\item $\displaystyle \gamma_{\stackrel{}{P}}(G \circ H)=1$ if and only if $\gamma(G)$=1 and $\gamma_{\stackrel{}{P}}(H)=1$.
	
	\item $\displaystyle \gamma_{\stackrel{}{P}}(G \circ H)=2$ if and only if 
	either $\gamma(G)=2$ and $\gamma_{\stackrel{}{P}}(H)=1$
	or $diam(\overline{G})>2$ and $\gamma_{\stackrel{}{P}}(H)>1$.
	
	\end{itemize}
\end{cor}


The vertex set of the \emph{Cartesian product} $G \Box H$ of graphs $G$ and $H$ is $V(G) \times V (H)$. 
Let  $u = (g, h)$ and $v = (g', h')$ be a pair of vertices of $V (G) \times V (H)$.
Vertices $u$ and $v$ are adjacent in the \emph{Cartesian product} $G \Box H$ if 
either $g = g'$ and $hh' \in E(H)$, or $h = h'$ and $gg' \in E(G)$.

While a complete classification of graphs $G$ for which  $\gamma_p(G) = 1$ is not known yet and  it is certainly far for being simple, several authors were able to solve this problem  for the Cartesian product of two graphs. 
Before showing this  result, we define a graph operation. 
The graph obtained from $G$ and $H$ by {\it amalgamating} two vertices $g \in V (G)$ and $h \in V (H)$ has vertex set 
$V (G)\cup (V(H)\setminus\{h\})$ such that the subgraphs induced by $V(G)$ and $(V(H)\setminus\{h\})\cup\{g\}$ are $G$ and $H$, respectively.

\begin{theorem}[\cite{koso16,soko18,vavi16}]
Let G and H be two nontrivial graphs such that $\gamma(G) \le \gamma (H)$. 
Then, $\gamma_p(G \Box H) = 1$ if and only if either
\begin{enumerate}[{\rm (1)}]

\item $G$ and $H$ each has order at least four, $\gamma(G)=1$ and $H$ is a path, or

\item $G$ is either $P_2$ or $P_3$ and $H$ can be obtained by amalgamating any vertex of a graph, say $D$, with  
$\gamma(D) = 1$ and an end vertex of $P_n$ with $n \ge 1$, or

\item $G \cong C_3$ and the $H$ is a path.

\end{enumerate}
\end{theorem}

\begin{theorem} [\cite{dohe06}]\label{cartpr1}
Let $1\le m \le n$. Then,
 $\displaystyle \gamma_{\stackrel{}{P}}(P_m \Box P_n)=
\left\{
\begin{array}{lll}
      {\displaystyle \lceil \frac{m+1}{4} \rceil}, & {\rm if} & m \equiv 4 ({\rm mod}\, 8) \\
      
      {\displaystyle \lceil \frac{m}{4} \rceil}, &  & {\rm otherwise.} \\
\end{array} 
\right. 
$
\end{theorem}

\begin{cor}
Let $1\le m \le n$. Then,
 $\displaystyle \gamma_{\stackrel{}{P}}(P_m \Box P_n)=
\left\{
\begin{array}{lll}
      1, & {\rm if} & 1 \le m \le  3 \\
      
      2, & {\rm if} & 4 \le  m \le 8. \\
\end{array} 
\right. 
$
\end{cor}

\begin{theorem}[\cite{vavi16}]
Let $G$ and $H$ be two  graphs. 

\begin{enumerate}[{\rm (1)}]

\item If $\gamma(H)=1$, then $\gp(G \Box H) \le  Z(G)$.

\item If $H\cong P_n$, then $\gamma_{P}(G \Box P_{n}) \leq \gamma(G)$.

\end{enumerate}

\end{theorem}

\begin{cor}
Let $G$ and $H$ be two  graphs of order at least 4. 
If $G$ is outerplanar, $P(G)=2$ and $ \gamma(H)=1$, then $\gp(G\Box H) = 2$.
\end{cor}

\begin{theorem}[\cite{koso16}]
Let $G$ and $H$ be two  graphs. 

\begin{enumerate}[{\rm (1)}]

\item $\max\{\gp(G),\gp(H)\} \le \gp (G\Box H)$.

\item If $H$ is a tree $T$, then $\gp(G)\cdot\gp(T) \le \gp (G\Box T)$.

\end{enumerate}
 
\end{theorem}

\begin{cor}\label{cppaths}
For any  graph $G$, $\gamma_{P}(G) \le \gamma_{P}(G \Box P_{n})  \le   \gamma(G) $. 
In particular, 
\begin{itemize}
\item $\gp(G) \le \gp(G\Box P_2) \le \min\{\gamma(G),Z(G)\}$.
\item If $ \gamma_{P}(G)=\gamma(G) $, then $ \gamma_{P}(G \Box P_{n}) = \gamma(G) $.
\end{itemize}
\end{cor}

\begin{prop}
	For any  graph $G$,
if $\gp(G) =1$, then $\gp(G\Box P_2) \le 2$.
\end{prop}
\begin{proof}
If $V(P_2)=\{v_1,v_2\}$, then, $V(G\Box P_2)=V_1\cup V_2=\{(x,v_1): x \in V(G)\}\cup\{(x,v_2):x\in V(G)\}$.
Let $\{u\}$ a $\gp$-set of $G$.
Take $S=\{(u,v_1),(u,v_2)\}$.
Notice that $\{(u,v_1)\}$ is a $\gp$-set of $G_1=G[V(G)\times\{v_1\}]$ and $\{(u,v_2)\}$ is a $\gp$-set of $G_2=G[V(G)\times\{v_2\}]$.
Hence, $S$ is power dominating set of $G\Box P_2$, i.e., $\gp(G\Box P_2) \le 2$.

\end{proof}

\begin{cor}
Let be a graph such that $\gp(G) =1$.
Then, $\gp(G\Box P_2)=2$ if and only if $G$ can not be obtained by amalgamating any vertex of a graph, say $D$, with  
$\gamma(D) = 1$ and an end vertex of $P_n$ with $n \ge 1$.
\end{cor}

Observe that, according to Corollary \ref{cppaths}, if $\gp(G) \ge 2$ and $\gp(G\Box P_2) =2$, then $\gp(G) =2$.
Nevertheless, the converse is not true.
For example, if we consider the  graph $ G $ displayed in Figure \ref{fig8}, it is easy to check that  that $ \gp(G)=2 $ and $\gp(G\Box P_2) =3$.

\begin{figure}[!h]
	\centerline{\includegraphics[height=5cm]{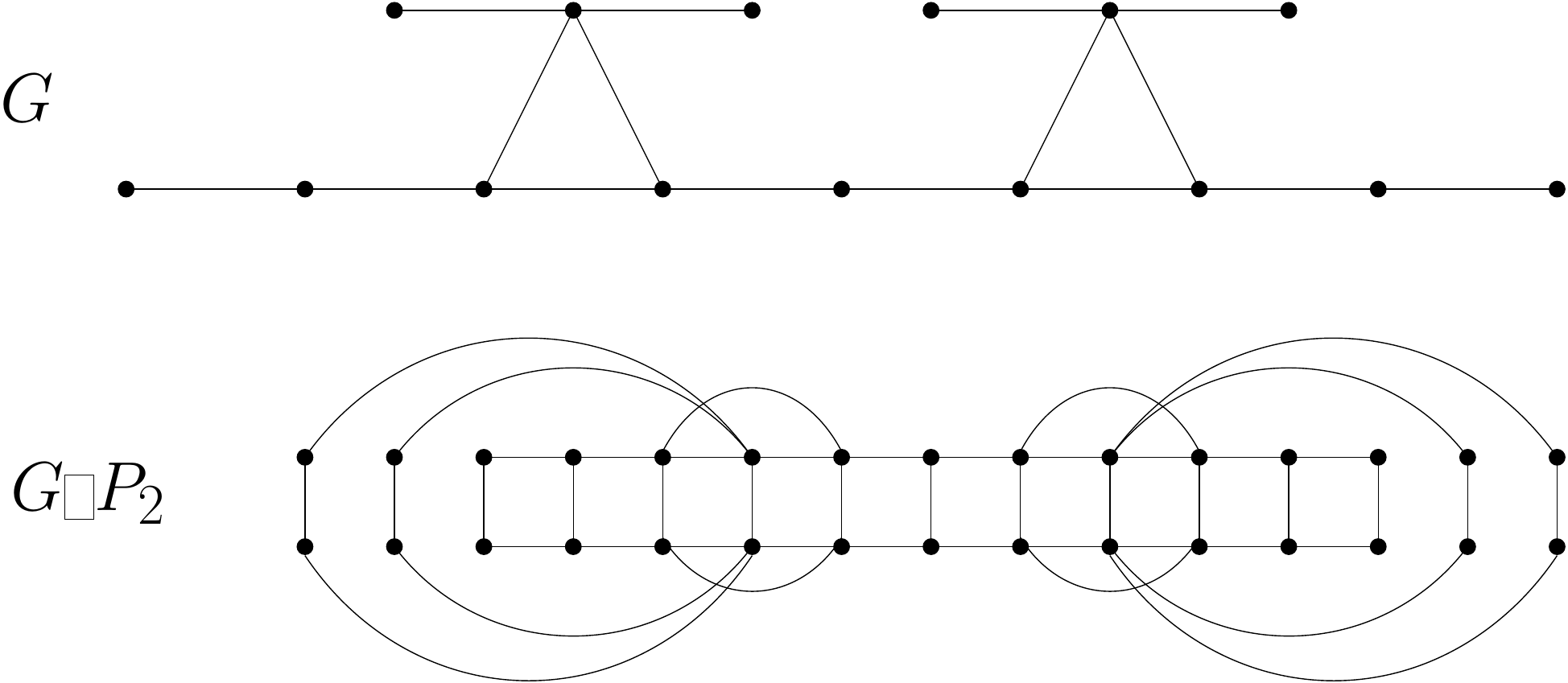}}
	\caption{$ \gp(G)=2 $ and $\gp(G\Box P_2) =3$.}
	\label{fig8}
\end{figure}

\newpage

\end{document}